\documentclass[11pt]{amsart}
\usepackage{graphicx}
\usepackage{amssymb}
\usepackage{epstopdf}
\usepackage{xcolor}
\usepackage{mathtools}
\usepackage[pagebackref,colorlinks,linkcolor=red,citecolor=blue,urlcolor=blue,hypertexnames=true]{hyperref}
\usepackage[pagebackref,hypertexnames=true]{hyperref}
\usepackage{enumerate}

\renewcommand{\le}{\leqslant}
\renewcommand{\leq}{\leqslant}
\renewcommand{\ge}{\geqslant}
\renewcommand{\geq}{\geqslant}

\newcommand{\SOTh}{\mathrm{SOT}\text{-}}

\newcommand{\N}{\mathbb{N}}

\newcommand{\cE}{\mathcal{E}}

\newcommand{\eps}{\varepsilon}

\newcommand{\er}{\mathbb{R}}

\newcommand{\cstu}{\mathrm{C}^*_u}
\newcommand{\csql}{\mathrm{C}^*_{ql}}
\newcommand{\csts}{\mathrm{C}^*_s}
\newcommand{\roeq}{\mathrm{Q}^*_u}
\newtheorem*{rigprob*}{Rigidity Problem for uniform Roe Algebras}
\newtheorem*{rigprobcorona*}{Rigidity Problem for uniform Roe Coronas}

\newcommand{\cst}{\mathrm{C}^*}
\newcommand{\cstar}{$\mathrm{C}^*$}

\newcommand{\cF}{\mathcal{F}}
\newcommand{\cP}{\mathcal{P}}
\newcommand{\bbN}{\mathbb{N}}
\newcommand{\cB}{\mathcal{B}}
\newcommand{\cK}{\mathcal{K}}

\newcommand{\brx}{B_r(x)}

\newtheorem{theorem}{Theorem}[section]
\newtheorem*{theorem*}{Theorem}
\newtheorem{proposition}[theorem]{Proposition}
\newtheorem{problem}[theorem]{Problem}
\newtheorem*{proposition*}{Proposition}
\newtheorem{lemma}[theorem]{Lemma}
\newtheorem*{lemma*}{Lemma}
\newtheorem{corollary}[theorem]{Corollary}
\newtheorem*{corollary*}{Corollar}

\newtheorem*{fact*}{Fact}
\theoremstyle{definition}
\newtheorem{definition}[theorem]{Definition}
\newtheorem*{acknowledgments}{Acknowledgments}
\newtheorem*{definition*}{Definition}
\newtheorem{claim}[theorem]{Claim}
\newtheorem*{claim*}{Claim}

\newtheorem*{conjecture*}{Conjecture}

\theoremstyle{remark}

\newtheorem*{example*}{Example}
\newtheorem{remark}[theorem]{Remark}
\newtheorem*{remark*}{Remark}

\newtheorem*{note*}{Note}
\newtheorem*{question*}{Question}

\newcommand{\norm}[1]{\left\lVert #1 \right\rVert}

\newcommand{\Manoa}{M\=anoa}
\newcommand{\Hawaii}{Hawai\kern.05em`\kern.05em\relax i}

\DeclareMathOperator{\supp}{supp}

\DeclareMathOperator{\propg}{prop}

\DeclareMathOperator{\rank}{rank}

\DeclareMathOperator{\diam}{diam}

\numberwithin{equation}{section}

\newcounter{my_enumerate_counter}
\newcommand{\pushcounter}{\setcounter{my_enumerate_counter}{\value{enumi}}}
\newcommand{\popcounter}{\setcounter{enumi}{\value{my_enumerate_counter}}}

\newcommand{\ulf}{\textrm{uniformly locally finite\ }}
\newcommand{\ulfx}{\textrm{uniformly locally finite}}

\DeclareMathOperator{\conv}{conv}

\title[{Uniform Roe algebras are rigid}]{Uniform Roe algebras of uniformly locally finite metric spaces are rigid}

\date{\today} 

\author[Baudier]{Florent P. Baudier}
\address[F. P. Baudier]{Texas A\&M University, Department of Mathematics, College Station, TX 77843-3368, USA} 
 \email{florent@math.tamu.edu}
 \urladdr{https://www.math.tamu.edu/~florent/}
\author[Braga]{Bruno M. Braga}
\address[B. M. Braga]{University of Virginia, 141 Cabell Drive, Kerchof Hall, P.O. Box 400137, Charlottesville, USA}
\email{demendoncabraga@gmail.com}
\urladdr{https://sites.google.com/site/demendoncabraga}
 
\author[Farah]{Ilijas Farah}
\address[I. Farah]{Department of Mathematics and Statistics\\
York University\\
4700 Keele Street\\
North York, Ontario\\ Canada, M3J 1P3\\
and 
Matemati\v cki Institut SANU\\
Kneza Mihaila 36\\
11\,000 Beograd, p.p. 367\\
Serbia}
\email{ifarah@yorku.ca}
\urladdr{https://ifarah.mathstats.yorku.ca}

\author[Khukhro]{Ana Khukhro}
\address[A. Khukhro]{Murray Edwards College\\
Huntingdon Road\\
University of Cambridge\\
Cambridge\\
CB3 0DF, United Kingdom}
\email{ak467@cam.ac.uk}
\urladdr{https://sites.google.com/site/anakhukhromath/}

\author[Vignati]{Alessandro Vignati}
\address[A. Vignati]{
Institut de Math\'ematiques de Jussieu (IMJ-PRG)\\
Universit\'e Paris Cit\'e\\
B\^atiment Sophie Germain\\
8 Place Aur\'elie Nemours \\ 75013 Paris, France}
\email{alessandro.vignati@imj-prg.fr}
\urladdr{http://www.automorph.net/avignati}

\author[Willett]{Rufus Willett}
\address[R. Willett]{University of \Hawaii~at \Manoa, 2565 McCarthy Mall, Keller 401A, Honolulu, HI 96816, USA} 
\email{rufus@math.hawaii.edu}
\urladdr{https://math.hawaii.edu/~rufus/}

\subjclass[2010]{}

\begin{document}

\maketitle

\begin{abstract}
We show that if $X$ and $Y$ are uniformly locally finite metric spaces whose uniform Roe algebras, $\cstu(X)$ and $\cstu(Y)$, are isomorphic as \cstar-algebras, then $X$ and $Y$ are coarsely equivalent metric spaces. Moreover, we show that coarse equivalence between $X$ and $Y$ is equivalent to Morita equivalence between $\cstu(X)$ and $\cstu(Y)$. As an application, we obtain that if $\Gamma$ and $\Lambda$ are finitely generated groups, then the crossed products $\ell_\infty(\Gamma)\rtimes_r\Gamma$ and $ \ell_\infty(\Lambda)\rtimes_r\Lambda$ are isomorphic if and only if $\Gamma$ and $\Lambda$ are bi-Lipschitz equivalent. 
\end{abstract}

\section{Introduction}\label{SectionIntro}

Coarse geometry is the study of metric spaces when one forgets about the small scale structure and focuses only on large scales. For example, this philosophy underlies much of geometric group theory. As the local structure of a space is irrelevant, one typically assumes that the spaces one is working with are discrete: we will focus here on \emph{uniformly locally finite}\footnote{Also called \emph{bounded geometry} metric spaces in the literature.} metric spaces $(X,d_X)$, meaning that $\sup_{x\in X}|\brx|<\infty$ for all $r>0$, where $|\brx|$ is the cardinality of the closed ball in $X$ of radius $r$ centered at $x$. Typical examples that are important for applications are finitely generated groups with word metrics, and discretizations of non-discrete spaces such as Riemannian manifolds. There is a natural \emph{coarse category} of metric spaces considered from a large-scale point of view, and the isomorphisms in this category are called \emph{coarse equivalences}.

Here is the formal definition. Given metric spaces $(X,d_X)$ and $(Y,d_Y)$, a map $f\colon X\to Y$ is \emph{coarse} if for all $r>0$ there is $s>0$ so that 
\[d_X(x,x')\leq r\ \text{ implies } \ d_Y (f(x),f(x'))\leq s\]
for all $x,x'\in X$. If $f:X\to Y$ and $g:Y\to X$ are coarse and 
\[ \sup_{x\in X}d_X(x,g(f(x)))<\infty \ \text{ and }\ \sup_{y\in Y}d_Y (y,f(g(y)))<\infty\]
then $f$ and $g$ are called \emph{mutual coarse inverses}, each of $f$ and $g$ is called a \emph{coarse equivalence}, and $X$ and $Y$ are said to be \emph{coarsely equivalent}. 

Associated to the large-scale structure of a uniformly locally finite metric space is a \cstar-algebra, i.e., a norm-closed and adjoint-closed algebra of bounded operators on a complex Hilbert space, called the \emph{uniform Roe algebra of $X$} and denoted by $\cstu(X)$. Prototypical versions of this \cstar-algebra were introduced by Roe \cite{Roe:1988qy} for index-theoretic purposes. The theory was consolidated in the 1990s by Roe, Yu and others, and uniform Roe algebras have since found applications in index theory (for example, \cite{Spakula:2009tg,Engel:2018vm}), \cstar-algebra theory (for example, \cite{Rordam:2010kx,Li:2017ac}), single operator theory (for example, \cite{Rabinovich:2004xe,Spakula:2014aa}), topological dynamics (for example, \cite{Kellerhals:2013aa,Brodzki:2015kb}), and mathematical physics (for example, \cite{Cedzich:2018wx,Kubota2017}). 

Here is the formal definition. For a metric space $(X,d_X)$, the \emph{propagation} of an $X$-by-$X$ matrix $a=[a_{xy}]$ of complex numbers is
\[
\propg(a) \coloneqq \sup\{d_X(x,y)\mid a_{xy}\neq 0\}\in [0,\infty].
\]
If $a=[a_{xy}]$ has finite propagation and uniformly bounded entries, then $a$ canonically induces a bounded operator on the Hilbert space $\ell_2(X)$ as long as $(X,d_X)$ is uniformly locally finite. For any such $(X,d_X)$, the operators with finite propagation form a $*$-algebra, and $\cstu(X)$ is the \cstar-algebra defined as the norm closure of this $*$-algebra.

For many applications of uniform Roe algebras, one wants to know how much of the underlying metric geometry is remembered by $\cstu(X)$. This leads to the foundational question below.

\begin{problem}[Rigidity of uniform Roe algebras]
If the uniform Roe algebras of uniformly locally finite metric spaces are $*$-isomorphic, are the underlying metric spaces coarsely equivalent? \label{ProblemRig}
\end{problem}

Recently, the rigidity problem for uniform Roe algebras has been extensively studied. For example: \cite{SpakulaWillett2013AdvMath} started this study; \cite{BragaFarah2018Trans} introduced several new ideas that are relevant for this paper; and \cite{LiSpakulaZhang2020} represents the most recent developments before this paper. All of these papers (and others) give positive answers to Problem~\ref{ProblemRig} in the presence of additional geometric conditions on the underlying metric spaces.

\subsection{Main results}

In this paper, we give an \emph{unconditional} positive answer to the rigidity problem. 

\begin{theorem} \label{T1}
Let $X$ and $Y$ be uniformly locally finite metric spaces. If $\cstu(X)$ and $\cstu(Y)$ are $*$-isomorphic, then $X$ and $Y$ are coarsely equivalent. 
\end{theorem}

We also obtain the analogue of this theorem for the $C^*$-algebra of `quasi-local' operators: see Theorem~\ref{ql the} below. We will discuss an outline of the proof in Section~\ref{Subsecproof} below. For now, let us focus on some applications and elaborations.

A first application of Theorem~\ref{T1} regards groups. Associated to an action of a group $\Gamma$ on a compact topological space $K$, there is a \cstar-algebra \emph{crossed product} $C(K)\rtimes_r\Gamma$ that models the underlying dynamics. In particular, one can do this when $K=\beta\Gamma$, the \v{C}ech-Stone compactification of $\Gamma$, which is the universal compact $\Gamma$-space in some sense. If $\Gamma$ is discrete, $C(\beta \Gamma)$ naturally identifies with $\ell_\infty(\Gamma)$, so we get the crossed product $\ell_\infty(\Gamma)\rtimes_r\Gamma$. If $\Gamma$ is a finitely generated group, then it becomes a uniformly locally finite metric space when equipped with a word metric. The uniform Roe algebra of $\Gamma$ then identifies with the \cstar-algebra crossed product $\ell_\infty(\Gamma)\rtimes_r\Gamma$ discussed above, i.e., there is a canonical $*$-isomorphism $\cstu(\Gamma)\cong \ell_\infty(\Gamma)\rtimes_r\Gamma$ (see \cite[Proposition 5.1.3]{BrownOzawa08}).

The following result is of interest in pure \cstar-algebra theory and topological dynamics (see Corollary~\ref{bij cor 2} below for a more general statement).

\begin{corollary}\label{group cor}
\label{T1CorGroups}
Let $\Gamma$ and $\Lambda$ be finitely generated groups. The following are equivalent:
\begin{enumerate}[(1)]
\item With any choice of word metrics, $\Gamma$ and $\Lambda$ are bi-Lipschitz equivalent.\footnote{Metric spaces $(X,d_X)$ and $(Y,d_Y)$ are bi-Lipschitz equivalent if there is a bijection $f\colon X \to Y$ such that $f$ and $f^{-1}$ are Lipschitz.
}
\item The \cstar-algebras $\ell_\infty(\Gamma)\rtimes_r\Gamma$ and $\ell_\infty(\Lambda)\rtimes_r\Lambda$ are $*$-isomorphic. 
\end{enumerate} 
\end{corollary}

Our next main result concerns \emph{Morita equivalence}. This is a notion of isomorphism for \cstar-algebras that is a little weaker than $*$-isomorphism. Roughly, it says that the \cstar-algebras involved are $*$-isomorphic `up to multiplicity', and is typically considered the `correct' notion of isomorphism for \cstar-algebras in noncommutative geometry. That Morita equivalence of uniform Roe algebras is connected to coarse equivalence of the underlying spaces seems to have been guessed at by Gromov in the early 90s \cite[page 263]{Gromov93}. Brodzki, Niblo, and Wright \cite[Theorem 4]{BNW07} subsequently showed that coarse equivalence of uniformly locally finite metric spaces implies Morita equivalence of their uniform Roe algebras. Our methods allow us to obtain that the converse also holds. 

\begin{theorem}\label{TMorita}
Let $X$ and $Y$ be uniformly locally finite metric spaces. The following are equivalent:
\begin{enumerate}[(1)]
\item\label{Item1TMorita} $X$ and $Y$ are coarsely equivalent.
\item\label{Item2TMorita} $\cstu(X)$ and $\cstu(Y)$ are Morita equivalent.
\end{enumerate}
\end{theorem}

Our findings also allow us to remove the geometric assumptions from the main result of \cite{BragaFarahVignati2018}. Precisely, if $X$ is a uniformly locally finite metric space then the compact operators $\mathcal{K}(\ell_2(X))$ form an ideal in $\cstu(X)$. The associated quotient is the \emph{uniform Roe corona of $X$}, denoted by $\roeq(X)$. In \cite{BragaFarahVignati2018}, the authors investigate whether rigidity also holds given the weaker assumption of isomorphism between uniform Roe coronas. In this paper we obtain the following:
 
\begin{theorem} \label{T1.Corona}
Let $X$ and $Y$ be uniformly locally finite metric spaces. If $\roeq(X)$ and $\roeq(Y)$ are $*$-isomorphic and one assumes appropriate forcing axioms\footnote{For the set theorist reader, this result is a theorem in $\mathrm{ZFC}+\mathrm{OCA}_{\mathrm{T}}+\mathrm{MA}_{\aleph_1}$.}, then $X$ and $Y$ are coarsely equivalent. 
\end{theorem}

\begin{remark}
There is also a ``non-uniform'' Roe algebra $C^*(X)$ associated to a bounded geometry metric space: roughly, one defines this $C^*$-algebra by taking the matrix entries to be compact operators, as opposed to complex numbers.  Our methods do not obviously apply to Roe algebras: we refer to Remark~\ref{no roe rem} for detailed definitions, and discussion of the issues that arise.  In the case of Roe algebras, the state of the art rigidity result is obtained in the work of Li, \v{S}pakula, and Zhang \cite{LiSpakulaZhang2020} mentioned above; the current paper offers no real improvements.
\end{remark}

\subsection{The road to rigidity}\label{Subsecproof}

We now discuss our methods of proof in more detail. If $H$ is a Hilbert space then $\cB(H)$ denotes the \cstar-algebra of all bounded operators on $H$. The strong operator topology on $\cB(H)$ is the topology of pointwise convergence on $\cB(H)$. We write ``SOT'' as an abbreviation for ``strong operator topology'' and ``$\SOTh\sum$'' for a sum that converges in the strong operator topology. 

As already noted above, the \cstar-algebra of compact operators $\mathcal K(\ell_2(X))$ is an ideal in $\cstu(X)$, and in fact is the unique minimal ideal. As a result, a $*$-isomorphism between uniform Roe algebras of \ulf metric spaces sends compact operators to compact operators. Isomorphisms of the compact operators must be ``spatially implemented'', i.e., given by conjugation by an isometric isomorphism between the corresponding Hilbert spaces (see for example \cite[Corollary 4.1.8]{Dixmier:1977vl}). From this discussion, it is not difficult to deduce the following result.


\begin{lemma}[{\cite[Lemma 3.1]{SpakulaWillett2013AdvMath}}] \label{LemmaIsoImplementedUnitary}
Let $X$ and $Y$ be \ulf metric spaces and $\Phi\colon\cstu(X)\to \cstu(Y)$ be a $*$-isomorphism. Then there is an isometric isomorphism $u\colon\ell_2(X)\to \ell_2(Y)$ so that 
$\Phi(a)=uau^*$ for all $a\in \cstu(X)$. In particular, $\Phi$ is rank-preserving and continuous for the strong operator topology. \qed
\end{lemma}
The automatic SOT-continuity of isomorphisms between uniform Roe algebras will be very important for us.
In addition to this basic observation, our proof of Theorem~\ref{T1} has two main ingredients at its core: 
\begin{enumerate}[(I)]
\item \label{equiapprox} the ``equi-approximability'' of certain families of operators by operators with uniformly bounded propagation (see Lemma~\ref{LemmaUnifApprox});
\item \label{lowbound} a uniform lower bound on certain matrix coefficients.\footnote{This was formalized as \emph{rigidity} of a $*$-isomorphism in \cite[page 1008]{BragaFarah2018Trans}.}
\end{enumerate}

Let us first look at equi-approximability. We need a definition which quantifies how well a bounded operator can be approximated by a finite propagation operator.

\begin{definition} \label{Def.e-m-approximated} Let $X$ be a metric space, $\varepsilon>0$, and $r\geq0$. An operator $a$ in $\cB(\ell_2(X))$ is \emph{$\varepsilon$-$r$-approximable} if there exists $b\in \cB(\ell_2(X))$ with propagation at most $r$ such that $\|a-b\|\leq \varepsilon$. 
\end{definition}

The key ``equi-approximability lemma'' was obtained as an application of the Baire category theorem and diagonalization methods in \cite[Section 4]{BragaFarah2018Trans} (a weaker version appeared earlier in \cite[Lemma 3.2]{SpakulaWillett2013AdvMath}).

\begin{lemma}[{\cite[Lemma 4.9]{BragaFarah2018Trans}}]\label{equiapprox lem}
Let $X$ be a uniformly locally finite metric space and let $(a_n)_n$ be a sequence of operators so that $\SOTh\sum_{n\in M}a_n$ converges to an element of $\cstu(X)$ for all $M\subseteq \N$. Then for all $\eps>0$ there is $r>0$ so that $\SOTh\sum_{n\in M} a_n$ is $\varepsilon$-$r$-approximable for all $M\subseteq \N$.\qed \label{LemmaUnifApprox}
\end{lemma}

The second ingredient \eqref{lowbound} is a uniform lower bound on certain matrix entries. Given a set $X$, $(\delta_x)_{x\in X}$ denotes the standard orthonormal basis of $\ell_2(X)$ and, given $x,y\in X$, $e_{xy}$ denotes the rank 1 partial isometry sending $\delta_y$ to $\delta_x$. The current proofs of rigidity in the literature all follow a similar path: given a $*$-isomorphism $\Phi\colon\cstu(X)\to \cstu(Y)$, one uses some geometric property of $Y$ in order to ensure an inequality of the form 
\begin{equation}\label{low bound}
\inf_{x\in X}\sup_{y\in Y}\|\Phi(e_{xx})\delta_y\|>0.
\end{equation}
This inequality was first obtained in \cite[Lemma 4.6]{SpakulaWillett2013AdvMath} under the assumption of Yu's property A (see \cite[Definition 2.1]{Yu00}), which is an amenability-like property of metric spaces. 

The inequality in line \eqref{low bound} was then shown to hold under conditions on the absence of certain \emph{ghost operators} in \cite[Section 6]{BragaFarah2018Trans}: an operator $a=[a_{xy}]$ on $\ell_2(X)$ is a \emph{ghost} if $ \lim_{x,y\to\infty} a_{xy}=0$. Compact operators are easily seen to be ghosts, and we regard these as the trivial ghost operators. Property A is equivalent to the statement that all ghost operators are compact (\cite[Theorem 1.3]{RoeWillett14}), i.e., that there are no non-trivial ghosts. In \cite[Theorem 6.2]{BragaFarah2018Trans}, the inequality in line \eqref{low bound} was established under the absence of certain families of non-trivial ghost \emph{projections}, which is much weaker. Prior to this paper, the most general geometric condition that is sufficient to establish the inequality in line \eqref{low bound} also used ghostly ideas, and is due to Li, \v{S}pakula, and Zhang \cite[Theorem A]{LiSpakulaZhang2020}. Nonetheless, there are many examples where non-trivial ghosts exist, and that do not satisfy the Li--\v{S}pakula--Zhang condition.

The reason the condition in line \eqref{low bound} is useful is that it shows the existence of a map $f\colon X\to Y$ so that 
\[\inf_{x\in X}\|\Phi(e_{xx})\delta_{f(x)}\|>0.\]
The situation is symmetric, so that one also gets a map $g:Y\to X$ satisfying the same condition with the roles of $X$ and $Y$ reversed. Repeated use of the equi-approximability lemma (Lemma~\ref{LemmaUnifApprox} above) implies that the maps $f$ and $g$ are both coarse, and in fact mutual coarse inverses. We isolate the key point in the following proposition, the proof of which is contained in the proof of \cite[Theorem 4.1]{SpakulaWillett2013AdvMath} (see also \cite[Theorem 4.12]{BragaFarah2018Trans}).

\begin{proposition}\label{PropOnceMapsFoundWeAreGood}
Let $X$ and $Y$ be \ulf metric spaces and $\Phi\colon\cstu(X)\to\cstu(Y)$ be a $*$-isomorphism. If there are maps $f\colon X\to Y$ and $g\colon Y\to X$ so that $\inf_{x\in X}\|\Phi(e_{xx})\delta_{f(x)}\|>0$ and $\inf_{y\in Y}\|\Phi^{-1}(e_{yy})\delta_{g(y)}\|>0$, then $f$ and $g$ are mutual coarse inverses. \qed
\end{proposition}

The key new idea in the current paper establishes the inequality in line \eqref{low bound} \emph{unconditionally}. This is done by combining the equi-approximability lemma (Lemma~\ref{LemmaUnifApprox}) with a quantitative result on the approximate convexity of the range of a finite-dimensional, countably additive vector measure (see Lemma~\ref{lem:approx-infinite}).  We give two proofs of the latter fact: one due to the authors which is an application of the Shapley--Folkman theorem from economics, and the other suggested by a referee, which is based on Lindenstrauss' proof of Lyapunov's theorem.

\subsection{More rigidity}

We conclude this introduction with two other rigidity results. 

For the first, it has already been noted above that the previous partial solutions to Problem~\ref{ProblemRig} rely on conditions on the ideal of ghost operators in $\cstu(X)$. By its very definition, the ``ghost-ness'' of an operator is highly dependent on the choice of the orthonormal basis for $\ell_2(X)$. As such, it was unclear until now what happened to ghosts under $*$-isomorphisms. We solve this problem with the following result.
 
\begin{theorem}\label{TGhostsPreserved}
Every $*$-isomorphism between uniform Roe algebras of uniformly locally finite metric spaces sends ghost operators to ghost operators.
\end{theorem}

For the second result, we look at possibly non-metrizable coarse spaces. Just as topological spaces abstract the small scale structure of metric spaces, \emph{coarse spaces} abstract their large scale structure: see Section~\ref{SectionCoarse} for precise definitions. The definition of uniform Roe algebras extends to coarse spaces naturally, and rigidity of uniform Roe algebras of nonmetrizable coarse spaces has been studied in \cite{BragaFarah2018Trans,BragaFarahVignati2020AnnInstFour}. The proofs of our main results do not immediately extend to coarse spaces, since Lemma~\ref{LemmaUnifApprox} depends heavily on Baire categorical methods: these require coarse spaces to be metrizable (or at least small; see \cite[Definition~4.2]{BragaFarah2018Trans} and \cite[\S 8.5]{Fa:STCstar} for more information on the role of the Baire Category theorem), which translates to a countability condition on the associated coarse structure.

In the earlier work on rigidity, property A plays a key role, typically via the \emph{operator norm localization property} of Chen, Tessera, Wang, and Yu \cite[Section 2]{CTWY08}, which was shown to be equivalent to property A by Sako \cite{Sako14}. Our vector measure approach together with a new lemma inspired by Sako's work implies that the operator norm localization property holds for certain operators regardless of the geometry of the spaces (Lemma~\ref{LemmaSortOfONLForAnySpace}). We are thus able to establish the result below for general coarse spaces (see Section~\ref{SectionCoarse} for the definition of a \emph{coarse embedding}).

\begin{theorem} \label{T1CoarseSpEmb}
Let $(X,\cE)$ and $(Y,\cF)$ be uniformly locally finite coarse spaces, and suppose $(X,\cE)$ is metrizable. If $ \cstu(X)$ and $\cstu(Y)$ are $*$-iso\-mor\-phic, then $Y$ is countable and $X$ coarsely embeds into $Y$.
\end{theorem}

This result provides the first example of countable, coarse spaces without property A, whose uniform Roe algebras are not $*$-isomorphic to the uniform Roe algebra of any uniformly locally finite metric space. Indeed, any coarse space which contains no infinite metric space coarsely must satisfy this. In particular, this holds for $(\N,\cE_{\max})$, where $\cE_{\max}$ is the maximal uniformly locally finite coarse structure on $\N$, i.e., $E\in \cE$ if and only if the cardinality of the vertical and horizontal sections of $E$ are uniformly bounded. 

\begin{corollary} \label{T1CoarseSpEmb cor}
Let $\cE_{\max}$ be the maximal uniformly locally finite coarse structure on $\N$. Then $\cstu(\N,\cE_{\max})$ is not $*$-isomorphic to the uniform Roe algebra of any uniformly locally finite metric space.\qed
\end{corollary}

\section{Estimating the distance between the range of a vector measure and its convex hull} \label{SectionVectorMeasures}
 
In this section, we prove a quantitative estimate on the distance between the range of a finite-dimensional vector measure on $\cP(\N)$ --- the power set of $\N$ --- and its convex hull which will be crucial in what follows.  

A \emph{vector measure} is a function $\mu$ from a $\sigma$-algebra $\Sigma$ of sets into a Banach space which is countably additive, i.e., if $(A_n)_{n\in\bbN} $ is a sequence of disjoint sets in $\Sigma$, then $\mu(\bigcup_{n} A_n)=\sum_{n} \mu(A_n)$, where the sum converges in norm. 

The next lemma is the main result of this section.  For the statement, let us note that the norm on $\mathbb{R}^m$ is arbitrary, and that $\conv(S)$ denotes the convex hull of a subset of a vector space.
 
\begin{lemma}\label{lem:approx-infinite}
Let $m\in\N$ and $\mu \colon \cP(\N) \to (\mathbb{R}^m,\| \cdot \|)$ be a vector measure. Then, for all $v\in \conv(\mu[\cP(\N)])$ and $\varepsilon>0$, there exists a finite subset $F \subseteq \N$ such that 
 \begin{equation*}\label{eq:approx-infinite}
 \|\mu(F) - v\| \le \sup\{\|\mu(C)\|\mid C\subseteq \N, |C|\le m\} + \varepsilon.
 \end{equation*}
In particular, $\|\mu(F) -v\| \le m\sup_{n\in\N} \| \mu(\{n\})\| + \varepsilon$.
\end{lemma}

We offer two proofs of Lemma~\ref{lem:approx-infinite}. The first one was suggested by the referee, and is a modification of Lindenstrauss' proof of Lyapunov's convexity theorem: see~\cite{Lindenstrauss66}.

\begin{proof}[First proof of Lemma~\ref{lem:approx-infinite}] We will find a (possibly infinite) set $F$ such that $ \|\mu(F) - v\| \le \sup\{\|\mu(C)\|\mid C\subseteq \N, |C|\le m\}$. This suffices since $\mu(F)$ is the supremum of $\mu(G)$, with $G\subseteq F$ finite.   Throughout the proof, and contrary to our usual conventions in this paper, all $\ell_\infty$ and $\ell_1$ spaces are taken over $\mathbb{R}$, not $\mathbb{C}$.  

Identify $A\subseteq \N$ with the characteristic function of $A$, $\chi_A\in \ell_\infty(\N)$, and let $T$ be the linear extension of $\mu$ to $\text{span}\{\chi_A\mid A\subseteq \N\}\subseteq \ell_\infty(\N)$. We claim that $T$ is bounded and weak-$*$-continuous, whence in particular it extends to all of $\ell_\infty(\N)$. It suffices to show that the composition of $T$ with each one of the coordinate projections $\pi_j:\mathbb{R}^m\to\mathbb{R}$ has these properties. Fix $j\leq m$ and define $\mu_j \coloneqq\pi_j\circ \mu$, which we may think of as a function $\mu_j:\mathbb{N}\to \mathbb{R}$. 

Since $X_j:=\{n\mid \mu_j(\{n\})\geq 0\}$ and its complement explicitly provide the Hahn decomposition of $\mu_j$, the latter is absolutely summable.  As $\pi_j \circ T$ is the linear functional on $\ell_\infty(\mathbb{\N})$ given by pairing with $\mu_j\in \ell_1(\N)$, it is weak-$*$ continuous. 

Now, let $B\coloneqq\ell_\infty(\N)_{1,+}$ denote those $g\in \ell_\infty(\N)$ such that $0\leq g\leq 1$, which is a weak-$*$ compact set. Noting that $B$ contains all characteristic functions of subsets of $\N$, $T(B)$ is a convex set that contains $\mu(\mathcal{P}(\N))$. Let $v\in \text{conv}(\mu(\mathcal{P}(\N)))$ be arbitrary. Then $W_0:=T^{-1}(\{v\})\cap B$ is a non-empty, convex, weak-$*$ closed subset of $B$. Hence by the Krein-Milman theorem, $W_0$ has an extreme point, say $g$. 

We claim that the set $\{n\mid 0<g(n)<1\}$ has cardinality at most $m$. Otherwise, for some $\varepsilon>0$ the set $E=\{n\mid \varepsilon <g(n)<1-\varepsilon\}$ has cardinality greater than $m$. The rank-nullity theorem implies that the restriction of $T$ to $\ell_\infty(E)\subseteq \ell_\infty(\N)$ contains an element $h$ of norm one in its kernel. Then $g\pm \varepsilon h$ are distinct points in $W_0$ such that $g$ is their midpoint, contradicting extremity. Therefore $g$ differs from the characteristic function of $F \coloneqq\{n\mid g(n)=1\}$ on at most $m$ points, as required. 
\end{proof}

The second proof of Lemma~\ref{lem:approx-infinite} (our original proof) is an application of the Shapley--Folkman theorem. The insight provided by this theorem was one of the key factors that led to the proof of our main result. 

If $S_1,...,S_n$ are subsets of a vector space, their \emph{Minkowski sum} is
\[
\sum_{i=1}^n S_i\coloneqq \{s_1+\cdots +s_n\mid s_i\in S_i\}.
\]
It is a well-known elementary fact that the convex hull of a Minkowski sum is the Minkowski sum of the convex hulls. Precisely, given subsets $(S_i)_{i=1}^n$ of a vector space, we have 
 \begin{equation}\label{eq:MS-SM}
 \conv \Big( \sum_{i=1}^n S_i \Big) = \sum_{i=1}^n \conv(S_i).
 \end{equation}
The Shapley--Folkman theorem (see \cite[Appendix 2]{starr1969quasi}, or \cite{Zhou93} for a short proof) provides additional quantitative information about the nature of the decomposition in \eqref{eq:MS-SM} when the subsets are drawn from a finite-dimensional vector space. Precisely:

\begin{theorem}[Shapley--Folkman theorem]
\label{ThmShapleyFolkman}
Let $m\in\N$, $(S_i)_{i=1}^n$ be nonempty subsets of $\mathbb R^m$. Then each $v\in\conv ( \sum_{i=1}^n S_i )$ can be written as $v=\sum_{i=1}^nv_i$ where $v_i\in \conv(S_i)$ for all $i\in \{1,\ldots, n\}$, and so that 
\begin{equation*}
|\{i\in \{1,\ldots ,n\}\mid v_i\not\in S_i\}|\leq m. 
\end{equation*}
\end{theorem}

We now use the Shapley--Folkman theorem to prove Lemma~\ref{lem:approx-infinite} for finite sets.

 \begin{lemma}\label{lem:approx-finite}
Let $m\in\N$, $X$ be a finite set, and $\mu \colon \cP(X) \to (\mathbb{R}^m,\| \cdot \|)$ be a vector measure. Then, for all $v\in \conv(\mu[\cP(X)])$, there exists a subset $F \subseteq X$ such that 
 \begin{equation*}\label{eq:approx-finite}
 \|\mu(F) - v\| \le \max\{\|\mu(C)\|\mid C\subseteq X, |C|\le m\}.
 \end{equation*}
In particular, $\|\mu(F) - v\| \le m\max_{x\in X} \| \mu(\{x\})\|$.
 \end{lemma}

\begin{proof} 
By shrinking $X$, we may assume that $\mu(\{x\})\neq 0$ for all $x\in X$. For each $x\in X$, define $S_x \coloneqq \{0, \mu(\{x\}) \}\subseteq \mathbb{R}^m$. As $0\in S_x$, we have 
 \begin{equation*}
 \mu [\cP(X)] = \sum_{x\in X} S_x.
 \end{equation*}
Hence, if $v\in \conv(\mu [\cP(X)])$, then it follows from \eqref{eq:MS-SM} that there are $v_x\in \conv(S_x)$, for $x\in X$, such that $v = \sum_{x\in X} v_x$. By the Shapley--Folkman theorem, we may assume that the set $ C \coloneqq \{ x\in X \mid v_x \notin S_x\}$ has cardinality at most $m$. From \eqref{eq:MS-SM}, $\sum_{x\in C}\conv(S_x)$ is equal to the convex hull of the set $\sum_{x\in C}S_x=\{\mu(D)\mid D\subseteq C\}$. Therefore 
 \begin{equation*}
 \left\|\sum_{x\in C} v_x\right\|\le \max_{D\subseteq C} \|\mu(D)\|\le
 \max\{\|\mu(D)\|\mid D\subseteq X, |D|\le m\}. 
 \end{equation*}
 Define 
 \begin{equation*}
 S \coloneqq \{ x\in X \mid v_x = \mu(\{x\}) \}\ \text{ and }\ V \coloneqq \{ x\in X \mid v_x = 0\}.
 \end{equation*}
Note that $S \sqcup C\sqcup V = X$. Hence, 
 \begin{align*}
 \| \mu(S) - v \| & = \Bigg\| \mu(S) - \sum_{x\in X} v_x \Bigg\| \\
 &\leq \Bigg\| \mu(S)- \sum_{x \in S} v_x\Bigg\|+\Bigg\|\sum_{x \in C} v_x\Bigg\| + \Bigg\|\sum_{ x \in V} v_x\Bigg\|\\
 & \le\max\{\|\mu(D)\|\mid D\subseteq X, |D|\le m\}.
 \end{align*} 
 To conclude, it remains to note that $\|\mu(D)\|\leq |D|\max_{x\in D} \|\mu(\{x\})\|$ for all $D\subseteq X$. 
\end{proof}
 
Our second proof of Lemma~\ref{lem:approx-infinite} now follows by a simple approximation argument.

\begin{proof}[The second proof of Lemma~\ref{lem:approx-infinite}]
If $\sup_{n\in\N} \| \mu(\{n\}) \|\in \{0, \infty\}$, the result is trivial, so assume that $ \sup_{n\in\N} \| \mu(\{n\}) \| \in (0,\infty)$. Let $\varepsilon > 0$ and $v\in \conv(\mu[\cP(\N)])$. Then $ v = \sum_{i=1}^k \lambda_i \mu (N_i)$ for some $N_1,\dots, N_k \subseteq \N$ and $\lambda_1,\dots,\lambda_k \ge 0$ such that $\sum_{i=1}^k \lambda_i=1$. Pick finite subsets $A_1, \dots, A_k\subseteq \N$ so that $\| \mu(N_i) - \mu(A_i)\| < \varepsilon$ for all $1\leq i \le k$. Let $A \coloneqq \bigcup_{i=1}^k A_i$ and $\mu_A$ be the restriction of $\mu$ to $\cP(A)$. So, $A$ is finite and $v_A \coloneqq \sum_{i=1}^k \lambda_i \mu(A_i)$ belongs to $\conv(\mu_A[\cP(A)])$. By Lemma~\ref{lem:approx-finite}, there exists a (finite) subset $F \subseteq A $ such that 
\begin{align*}
\| \mu_A(F) - v_A\| 
&\leq \max\{\|\mu_A(C)\|\mid C\subseteq A, |C|\le m\}\\
&\leq \sup\{\|\mu(C)\|\mid C\subseteq \N, |C|\le m\}.
\end{align*}
Since $\mu(F) = \mu_A (F)$, we have that
 \begin{align*}
 \| \mu(F) - v \| & \leq \| \mu_A (F) - v_A \| + \| v_A - v \|\\
 & \le \sup\{\|\mu(C)\|\mid C\subseteq \N, |C|\le m\} + \sum_{i=1}^k \lambda_i \| \mu (A_i) - \mu(N_i)\|\\
 & \le \sup\{\|\mu(C)\|\mid C\subseteq \N, |C|\le m\} + \varepsilon, 
 \end{align*}
and the statement is proved. 
\end{proof}

\begin{remark}\label{RemarkVecMeas}
The celebrated Lyapunov convexity theorem \cite{Liapounoff40} states that the range of a finite-dimensional atomless vector measure is closed and convex. 
Our measures of interest are atomic, however, and the ranges of such measures are not necessarily convex. 
On the other hand, a theorem of Elton--Hill (see \cite[Theorem~1.2]{EltonHill87}) quantifies the distance between the range of a finite-dimensional vector measure and its convex hull in terms of the size of the atoms of the one-dimensional coordinate measures. In a less elementary and self-contained way, Lemma~\ref{lem:approx-infinite} can also be obtained as an application of the Elton--Hill theorem. 
\end{remark}
 
\section{Rigidity of uniform Roe algebras}\label{SectionMainTheorem}
 
This section contains the proofs of Theorems~\ref{T1} and~\ref{TGhostsPreserved}. The former could also be obtained as a corollary of Theorem~\ref{T1.Stable}. However, for expository reasons we chose to present the proof of Theorem~\ref{T1} first. 
 

 Lemma~\ref{LemmaMeasureURA.Gen} below is our main technical tool, and is of independent interest. The following simple observation about idempotents was inspired by the referee report; it is used near the end of the proof of Lemma~\ref{LemmaMeasureURA.Gen}, and considerably simplifies our original argument.

\begin{lemma} \label{LemmaReferee} 
Let $E$ be a Banach space, $p\in\mathcal \cB(E)$ be an idempotent bounded operator, and $v\in E$ be a vector. Let $\delta>0$. If $\norm{pv-\frac{1}{2}v}<\delta$, then $\norm{v}<2\|2p-1\|\delta$.

If moreover $E$ is a Hilbert space and $p$ is a (self-adjoint) projection, then $\|v\|<2\delta$.
\end{lemma}

\begin{proof}
Define $w\coloneqq \frac{1}{2}v\in E$, and $a\coloneqq 2p-1\in \cB(E)$.  Then $a^2=1$, whence
\[
\|w\|=\|a^2w\|\leq \|a\|\|aw\|=\|a\|\|2pw-w\|=\|a\|\|pv-w\|<\|a\|\delta,
\]
giving the general result.  In the Hilbert space case, we note that $a$ is unitary, so in particular $\|a\|=1$.
\end{proof}


 
\begin{lemma}\label{LemmaMeasureURA.Gen}
Let $\varepsilon,r> 0$ and $X$ be a \ulf metric space, and define $N_r\coloneqq \sup_{x\in X}|B_r(x)|$. Let $(p_n)_{n\in \N}$ be a sequence of projections in $\mathcal{B}(\ell_2(X))$ such that:
\begin{enumerate}[(1)]
\item \label{LemmaMeasureURA.Gen1} $\SOTh\sum_{n\in A}p_n$ is $\varepsilon$-$r$-approximable for all $A\subseteq \N$, and 
\item \label{LemmaMeasureURA.Gen2} $\SOTh\sum_{n\in\N}p_n=1_{\ell_2(X)}$.
\end{enumerate}
For each $x\in X$ and $\delta>0$, define $M(x,\delta) \coloneqq \{n\in\N\mid \|p_n\delta_x\|\geq\delta\}$. Then, if $\delta \leq \varepsilon/(2N_r)$, we have
\[ 
\inf_{x\in X} \Bigg\|\sum_{n\in M(x,\delta)}p_n\delta_x\Bigg\|\geq 1-4\varepsilon.
\]
\end{lemma}

\begin{proof} 
Fix $\varepsilon,r>0$ and $N_r$ as in the statement. For $A\subseteq \N$ define $p_A \coloneqq \SOTh\sum_{n\in A}p_n$. Notice that the condition $p_\N=1_{\ell_2(X)}$ forces the projections $(p_n)_n$ to be mutually orthogonal, so each $p_A$ is a projection. Fix $\delta\leq \varepsilon/(2N_r)$ and $x\in X$.  Define $M\coloneqq M(x,\delta)$ and $M'\coloneqq\mathbb N\setminus M$.

Now comes the crucial vector measure argument. Let $\pi\colon \ell_2(X)\to \ell_2(B_r(x))$ be the canonical orthogonal projection, which we identify with $\chi_{\brx}$.\footnote{Here we use the following standard notation: for $S\subseteq X$, we let $\chi_{S} \coloneqq \SOTh\sum_{x\in S}e_{xx}$, i.e., $\chi_S$ is the operator on $\ell_2(X)$ that projects onto the coordinates indexed by $S$.} 
We define a vector measure $\mu \colon \cP(M')\to \ell_2(B_r(x))$ by $\mu(A):=\pi p_A \delta_x$
for all $A\subseteq M' $ ($\mu$ is clearly countably additive). As $\| \mu (\{n\})\|= \|\pi p_n\delta_x \|\leq \|p_n\delta_x\|$, it follows from the definition of $M'$ that $\sup_{n\in M'}\| \mu (\{n\})\|<\delta$. Therefore, as $\mu(M')/2$ belongs to the convex hull of the range of $\mu$ and as $\dim_{\er}(\ell_2(B_r(x)))\leq 2N_r$, Lemma~\ref{lem:approx-infinite} gives a subset $A\subseteq M'$ such that 
\begin{equation}\label{eq:d-conv-range}
\Big\|\mu(A) - \frac{\mu(M')}{2}\Big\| < 2N_r\delta\leq \varepsilon
\end{equation}
On the other hand, 
\begin{equation}\label{epsr}
\Big\|\mu(A) - \frac{\mu(M')}{2}\Big\| =\Big\|\pi\Big(p_A-\frac{1}{2}p_{M'}\Big)\delta_x\Big\|=\Big\|\pi\Big(\frac{1}{2}p_A-\frac{1}{2}p_{M'\setminus A}\Big)\delta_x\Big\|.
\end{equation}
As $p_A$ and $p_{M'\setminus A}$ are $\varepsilon$-$r$-approximable, so is their convex combination $\frac{1}{2}p_A-\frac{1}{2}p_{M'\setminus A}$, whence $\|(1-\pi)(\frac{1}{2}p_A-\frac{1}{2}p_{M'\setminus A})\delta_x\|<\varepsilon$.
This last inequality, lines \eqref{eq:d-conv-range} and \eqref{epsr}, and the fact that $p_A=p_Ap_{M'}$ together imply that 
\begin{align*}
\Big\|p_Ap_{M'}\delta_x- \frac12 p_{M'}\delta_x\Big\|  & \leq \Big\|(1-\pi)\Big(\frac{1}{2}p_A-\frac{1}{2}p_{M'\setminus A}\Big)\delta_x\Big\|+\Big\|\mu(A) - \frac{\mu(M')}{2}\Big\| \\ & <2\varepsilon.
\end{align*}
Hence Lemma \ref{LemmaReferee} implies that $\norm{p_{M'}\delta_x}<4\varepsilon$.  Using assumption~\eqref{LemmaMeasureURA.Gen2}, we have that $p_M+p_{M'}=1$, so $\|p_{M}\delta_x\|\geq 1-4\varepsilon$. As $p_M=\sum_{n\in M(x,\delta)}p_n$, we have the desired inequality.
\end{proof}

Lemma~\ref{LemmaUnifApprox} and Lemma~\ref{LemmaMeasureURA.Gen} imply the following corollary.

\begin{corollary}\label{CorollaryMeasureURA}
Let $X$ be a \ulf metric space and let $(p_n)_{n\in \N}$ be a sequence of projections in $\mathcal{B}(\ell_2(X))$ such that 
\begin{enumerate}[(1)]
\item $\SOTh\sum_{n\in A}p_n\in \cstu(X)$ for all $A\subseteq \N$, and 
\item $\SOTh\sum_{n\in\N}p_n=1_{\ell_2(X)}$.
\end{enumerate}
Then, 
\[
\inf_{x\in X}\sup_{n\in\N} \|p_n\delta_x\|>0.
\]
\end{corollary}

\begin{proof}
Let $\varepsilon:=1/5$. Then Lemma~\ref{LemmaUnifApprox} implies there is $r$ such that the projection $\SOTh\sum_{n\in A}p_n$ is $\varepsilon$-$r$-approximable for all $A\subseteq \N$. Lemma~\ref{LemmaMeasureURA.Gen} implies in particular that for $\delta=1/(10N_r)$ and any $x\in X$, $M(x,\delta)$ is non-empty. Hence 
\[
\inf_{x\in X}\sup_{n\in\N} \|p_n\delta_x\|\geq 1/(10N_r). \eqno \qedhere
\]
\end{proof} 

\begin{proof}[Proof of Theorem~\ref{T1}]
Fix a $*$-isomorphism $\Phi\colon\cstu(X)\to \cstu(Y)$. By Lem\-ma~\ref{LemmaIsoImplementedUnitary}, $\Phi$ is strongly continuous, so $(\Phi(e_{xx}))_{x\in X}$ satisfies the conditions on the family $(p_n)_{n\in \N}$ from Corollary~\ref{CorollaryMeasureURA}. Therefore, there are $\delta>0$ and $g\colon Y\to X$ such that 
\[
\|\Phi^{-1}(e_{yy})\delta_{g(y)}\|= \| \Phi(e_{g(y)g(y)}) e_{yy}\|=\|\Phi(e_{g(y)g(y)})\delta_y\|>\delta
\]
for all $y\in Y$. 

Replacing $\delta$ by a smaller positive real if necessary, an argument analogous to the one above applied to $\Phi^{-1}\colon\cstu(Y)\to \cstu(X)$ gives us a map $f\colon X\to Y$ such that 
\[
\|\Phi (e_{xx})\delta_{f(x)}\|>\delta,
\]
for all $x\in X$. By Proposition~\ref{PropOnceMapsFoundWeAreGood}, $f$ is a coarse equivalence.
\end{proof}

In the remainder of this section, we give several related rigidity results that can be established using the techniques developed above.

\subsection{Quasi-local operators}

The first of our additional rigidity results concerns quasi-local operators as in the next definition.

\begin{definition}\label{ql def}
A bounded operator $a$ on $\ell_2(X)$ is \emph{$\varepsilon$-$r$-quasi-local} if whenever $A,B\subseteq X$ satisfy $d(A,B)>r$, we have $\|\chi_Aa\chi_B\|<\varepsilon$, and is \emph{quasi-local} if for all $\varepsilon>0$ there exists $r\geq 0$ such that $a$ is $\varepsilon$-$r$-quasi-local. The collection of all quasi-local operators forms a $\mathrm{C}^*$-algebra, denoted $\csql(X)$. 
\end{definition}

One has $ \cstu(X)\subseteq \csql(X)$, and \v{S}pakula-Zhang \cite[Theorem 3.3]{Spakula:2018aa} (building on techniques from \v{S}pakula-Tikuisis \cite{Spakula:2017aa}) have shown that this inclusion is the identity when $X$ has property A. In general, it is not known whether the inclusion $\cstu(X)\subseteq \csql(X)$ can be strict.

\begin{theorem}\label{ql the}
Let $X$ and $Y$ be uniformly locally finite metric spaces. If $\csql(X)$ and $\csql(Y)$ are $*$-isomorphic, then $X$ and $Y$ are coarsely equivalent. 
\end{theorem}

\begin{proof}
The proof is similar to that of Theorem~\ref{T1}, so we just give a sketch. The first step is to establish a quasi-local version of the equi-approximability lemma (Lemma~\ref{LemmaUnifApprox}). This says: ``if $(a_n)_n$ is a sequence of orthogonal operators on $\ell_2(X)$ so that $\SOTh\sum_{n\in M}a_n$ converges to an element of $\csql(X)$ for all $M\subseteq \N$, then for all $\varepsilon>0$ there is $r\geq 0$ such that for all $M\subseteq \N$, $\SOTh\sum_{n\in M}a_n$ is $\varepsilon$-$r$-quasi-local.'' This quasi-local equi-approximability lemma follows from a slight adaptation of \cite[Lemma 3.2]{SpakulaWillett2013AdvMath}. On the other hand, the proof of Lemma~\ref{LemmaMeasureURA.Gen} goes through verbatim if condition \eqref{LemmaMeasureURA.Gen1} from the statement is replaced with ``$\SOTh\sum_{n\in A}p_n$ is $\varepsilon$-$r$-quasi-local for all $A\subseteq \N$". This quasi-local Lemma~\ref{LemmaMeasureURA.Gen} and the quasi-local equi-approximability lemma imply that Corollary~\ref{CorollaryMeasureURA} holds with ``$\cstu(X)$'' replaced by ``$\csql(X)$''. Finally, the quasi-local equi-approximability lemma is enough to establish the analogue of Proposition~\ref{PropOnceMapsFoundWeAreGood} for an isomorphism $\Phi\colon \csql(X)\to \csql(Y)$ (this is essentially what is done in the original rigidity paper \cite[Theorem 4.1]{SpakulaWillett2013AdvMath}), which completes the proof.
\end{proof}

\subsection{Uniform Roe algebras on more general Banach spaces}

A second elaboration concerns Banach algebras of operators on Banach spaces equipped with an appropriate basis, as developed by the second author \cite{Braga2021JFA}.  This material requires some background on bases of Banach spaces: for the benefit of non-experts, we introduce the terminology we need (see e.g.\ \cite{AlbiacKalton2016} for more detailed background).

\begin{definition}\label{bases}
A sequence $(e_n)_{n\in \N}$ in a Banach space $E$ is a \emph{Schauder basis} if every $e\in E$ can be expressed uniquely as $e=\sum_n \lambda_n e_n$ for some scalars $(\lambda_n)_{n\in \N}$, where the sum converges in norm.

A Schauder basis $(e_n)_{n\in \mathbb{N}}$ is \emph{unconditional} if for every permutation $\pi:\N\to\N$ the sequence $(e_{\pi(n)})_{n\in \mathbb{N}}$ is a Schauder basis. If moreover for every permutation $\pi:\N\to\N$, $(e_{\pi(n)})_{n\in \mathbb{N}}$ is equivalent to $(e_n)_{n\in \mathbb{N}}$, i.e.\ the map $\sum_n\lambda_ne_n\in E\mapsto \sum_n\lambda_ne_{\pi(n)}\in E$ defines an isomorphism, the basis is said to be \emph{symmetric}.
\end{definition}

\begin{definition}\label{Ban Roe}
Let $E$ be a Banach space equipped with a fixed Schauder basis $\mathcal{E}=(e_n)_{n\in \N}$, and let $d_E$ be a fixed \ulf metric on $\N$.  For each $n\in \N$, let $e_n^*$ be the bounded\footnote{For a proof that $e_n^*$ is extends to a well-defined bounded linear functional, see for example \cite[Theorem 1.1.3]{AlbiacKalton2016}.} linear functional on $E$ determined by $e_n^*(e_m)\coloneqq \delta_{nm}$, where the right hand side is the Kronecker $\delta$-function.  

A bounded operator $a$ on $E$ has \emph{propagation at most $r$} if $e_n^*(ae_m)=0$ whenever $d_E(n,m)>r$.  The \emph{uniform Roe algebra} $\mathrm{B}_u(d_E,\cE)$ associated to the triple $(E,\cE,d_E)$ is the norm closure of the finite propagation operators inside the set $\cB(E)$ of all bounded operators on $E$. 
\end{definition}

The reason one takes $\N$ as the set underlying the metric space in Definition \ref{Ban Roe} is that the definition and properties of a Schauder basis require an ordering on the basis.  However, for the rigidity theorem below we assume our bases are symmetric, so this ends up being irrelevant: the reader could just replace $(\N,d_E)$ and $(\N,d_F)$ with any \ulf metric spaces  in the statement of the theorem.

\begin{theorem}\label{Ban rig the}
Let $(E,\cE,d_E)$ and $(F,\cF,d_F)$ be triples as in Definition \ref{Ban Roe} with $\cE$ and $\cF$ symmetric Schauder bases.  Assume that $B_u(d_E,\cE)$ and $B_u(d_F,\cF)$ are isomorphic as Banach algebras.  Then $(\N,d_E)$ and $(\N,d_F)$ are coarsely equivalent.  
\end{theorem}

\begin{proof}
As for Theorem \ref{ql the}, we just sketch the proof.   As in \cite[Proposition 3.1]{Braga2021JFA}, the absolute property of the unconditional basis (see e.g. \cite[Proposition 3.1.3]{AlbiacKalton2016}) implies that $B_u(d_E,\cE)$ contains a copy of $\ell_\infty(\N)$ acting as multiplication operators with respect to the basis $\cE$.    

One then needs an analogue of Lemma \ref{LemmaIsoImplementedUnitary}, i.e.\ that for a Banach algebra isomorphism $\Phi:B_u(d_E,\cE)\to B_u(d_F,\cF)$ there is an isomorphism (not necessarily isometric) $u:E\to F$ such that $\Phi(a)=uau^{-1}$ for all $a\in B_u(d_E,\cE)$.  In particular, this implies that $\Phi$ is rank-preserving and strongly continuous.  This is established in \cite[Lemma 7.5]{Braga2021JFA}.  The second ingredient we need is equi-approximabilty (i.e.\ an analogue of Lemma~\ref{LemmaUnifApprox}), which is established in \cite[Theorem 7.9 and Lemma 7.10]{Braga2021JFA}.  

Having got these ingredients, an analogue of the rigidity criterion in Proposition \ref{PropOnceMapsFoundWeAreGood} goes through with exactly the same proof. It remains to establish an analogue of Lemma \ref{LemmaMeasureURA.Gen}, which also goes through directly with the same proof as long as one has uniform bounds on the norms of $\chi_{A}$ for any $A\subseteq \N$; this last fact is true since an unconditional basis is suppression-unconditional (see e.g. \cite[Proposition 3.1.5]{AlbiacKalton2016}).  
\end{proof}

We note that Theorem \ref{Ban rig the} applies in particular when $E=F=\ell_p(\N)$, the basis is the canonical one, and we equip $\N$ with two different \ulf metrics $d_E$ and $d_F$.  In this case if $p\in [1,\infty)\setminus \{2\}$,  Theorem \ref{Ban rig the} was proved by  Chung and Li (see \cite[Theorem 1.7]{ChungLi2018}) under the stronger assumption  of the algebras being \emph{isometrically} isomorphic as Banach algebras (under this stronger assumption, the authors even conclude that the metric spaces are \emph{bijectively} coarse equivalent).  The method of Chung and Li is quite different: the key point is that for $p\in [1,\infty)\setminus \{2\}$ isometric self-isomorphisms of $\ell_p(\N)$ are given by permutations of the basis, and multiplication of the basis elements by scalars of modulus one: see \cite[Proposition 2.3]{ChungLi2018} for more details.

\subsection{Amenability and groups}

For the next result, we use the notion of amenability for \ulf metric spaces, as introduced by Block and Weinberger \cite[Section 3]{Block:1992qp}.  It extends to \ulf metric spaces the classical group-theoretic notion of amenability. 

\begin{corollary}\label{bij cor}
Let $X$ and $Y$ be \ulf metric spaces, and consider the following statements:
\begin{enumerate}[(1)]
\item\label{as:1} $X$ and $Y$ are coarsely equivalent via a bijective coarse equivalence. 
\item\label{as:2} The \cstar-algebras $\cstu(X)$ and $\cstu(Y)$ are isomorphic. 
\end{enumerate} 
Then \eqref{as:1} implies \eqref{as:2} in general, and \eqref{as:2} implies \eqref{as:1} if $X$ is non-amenable.
\end{corollary}

\begin{proof}
The implication from \eqref{as:1} to \eqref{as:2} is well-known (e.g., \cite[Theorem 8.1]{BragaFarah2018Trans}): if $f \colon X\to Y$ is a bijective coarse equivalence, then one defines a unitary isomorphism $u \colon \ell_2(X)\to \ell_2(Y)$ by the formula $u\delta_x \coloneqq \delta_{f(x)}$ for all $x\in X$, and direct checks show that $u\cstu(X)u^*=\cstu(Y)$. 

For the converse, suppose that $\cstu(X)$ and $\cstu(Y)$ are $*$-isomorphic, whence by Theorem~\ref{T1}, $X$ and $Y$ are coarsely equivalent. If $X$ is not amenable, then $X$ being coarsely equivalent to a \ulf space $Y$ implies that $X$ is bijectively coarsely equivalent to it as shown in \cite[Theorem 5.1]{WhiteWillett20} (the key idea is from \cite[Theorem 4.1]{Whyte99}). 
\end{proof}

If $X$ and $Y$ are countable groups, we can do better. The following result gives Corollary~\ref{group cor}.

\begin{corollary}\label{bij cor 2}
Let $\Gamma$ and $\Lambda$ be countable discrete groups equipped with \ulf metrics that are invariant under left translation\footnote{Any countable group admits such a metric, which is moreover unique up to bijective coarse equivalence (e.g., \cite[Proposition 2.3.3]{Willett09})}. Then the following are equivalent:
\begin{enumerate}[(1)]
\item\label{as:12} $\Gamma$ and $\Lambda$ are coarsely equivalent via a bijective coarse equivalence. 
\item\label{as:22} The \cstar-algebras $\ell_\infty(\Gamma)\rtimes_r\Gamma$ and $\ell_\infty(\Lambda)\rtimes_r\Lambda$ are isomorphic. 
\end{enumerate}
Moreover, if $\Gamma$ and $\Lambda$ are finitely generated and equipped with word metrics, then one can replace ``bijective coarse equivalence'' in \eqref{as:12} with ``bi-Lipschitz bijection''.
\end{corollary}

\begin{proof}
We use the well-known identification $\cstu(\Gamma)\cong \ell_\infty(\Gamma)\rtimes_r\Gamma$ (see \cite[Proposition 5.1.3]{BrownOzawa08}) to replace the crossed products in \eqref{as:22} with uniform Roe algebras.

As already noted in the proof of Corollary~\ref{bij cor}, \eqref{as:12} implies \eqref{as:22} in general, and \eqref{as:22} implies \eqref{as:12} when $\Gamma$ is non-amenable. On the other hand, if $\Gamma$ is amenable, then as it is a group it has property A by \cite[Lemma 6.2]{Willett09}. A $*$-isomorphism between uniform Roe algebras of uniformly locally finite metric spaces, one of which has property A, gives a bijective coarse equivalence between the underlying metric spaces (this was proved in \cite[Corollary 6.13]{WhiteWillett20} for metric spaces and in \cite[Theorem~1.3]{BragaFarahVignati2020AnnInstFour} for arbitrary coarse spaces).

Assume now that $\Gamma$ and $\Lambda$ are finitely generated and equipped with word metrics, and assume that $\cstu(\Gamma)\cong \cstu(\Lambda)$. Using our discussion so far, there is a bijective coarse equivalence $f\colon \Gamma\to \Lambda$. As $\Gamma$ and $\Lambda$ are finitely generated, it is straightforward to check that they are quasi-geodesic in the sense of \cite[Definition 1.4.10]{NowakYu_book12}. Hence, $f$ is a quasi-isometry (cf. \cite[Proposition A.3]{GuentnerKaminker04} or \cite[Corollary 1.4.14]{NowakYu_book12}). As $\inf_{\gamma\neq \gamma'\in \Gamma}d_\Gamma(\gamma,\gamma')=1$ and $\inf_{\lambda\neq \lambda'\in \Lambda}d_\Lambda(\lambda,\lambda')=1$, a bijective quasi-isometry is automatically bi-Lipschitz.
\end{proof}

\begin{remark}\label{no bij rem}
We do not know whether \eqref{as:1} and \eqref{as:2} from Corollary~\ref{bij cor} are equivalent for \ulf metric spaces in general: a counterexample, if it exists, would have to be a pair of amenable, \ulf metric spaces, neither of which has property A. Many such examples exist: for example, any expander defines an amenable, \ulf metric space without property A.
\end{remark}

\subsection{Uniform Roe coronas} We now give the proof of Theorem~\ref{T1.Corona}.  Recall that the \emph{uniform Roe corona}, $\roeq(X)$,  of $X$ is the quotient of the uniform Roe algebra $\cstu(X)$ of $X$ by the ideal of compact operators $\mathcal K(\ell_2(X))$. 

 \begin{proof}[Proof of Theorem~\ref{T1.Corona}] Suppose that $X$ and $Y$ are uniformly locally finite metric spaces and that $\roeq(X)$ and $\roeq(Y)$ are $*$-isomorphic.  We will prove that  $\mathrm{OCA}_{\mathrm{T}}$ and $\mathrm{MA}_{\aleph_1}$ together imply that   $X$ and $Y$ are coarsely equivalent. 

 Let $\Lambda\colon \roeq(X)\to \roeq(Y)$ be a $*$-isomorphism, and let $\pi_X\colon \cstu(X)\to \roeq(X)$ and $\pi_Y\colon \cstu(Y)\to \roeq(Y)$ be the canonical projections. By \cite[Theorem 1.5]{BragaFarahVignati2018}, $\Lambda$ and $\Lambda^{-1}$ are \emph{liftable on the diagonals} in the sense of \cite[Definition 1.4(2)]{BragaFarahVignati2018}, i.e., there are strongly continuous $*$-homomorphisms $\Phi\colon \ell_\infty(X)\to \cstu(Y)$ and $\Psi\colon \ell_\infty(Y)\to \cstu(X)$ such that 
 \[
 \Lambda(\pi_X(a))=\pi_Y(\Phi(a))\quad \text{and}\quad\Lambda^{-1}(\pi_Y(b))=\pi_X(\Psi(b))
 \]
 for all $a\in \ell_\infty(X)$ and all $b\in \ell_\infty(Y)$.
 
 \begin{claim}
 There are cofinite subsets $X'\subseteq X$ and $Y'\subseteq Y$ such that \[\inf_{x\in X'}\sup_{y\in Y}\|\Phi(e_{xx})\delta_y\|>0\text{ and }\inf_{y\in Y'}\sup_{x\in X}\|\Psi(e_{yy})\delta_x\|>0.\]
 \end{claim}
 
\begin{proof}
By symmetry, it is enough to show that the result holds for $\Phi$. For that, define $p \coloneqq 1_{\ell_2(X)}-\Psi(1_{\ell_2(Y)})$. Then, as $\Lambda(\pi_X(\Psi(1_{\ell_2(Y)})))= \pi_Y(1_{\ell_2(Y)})$, it follows that $\Lambda(\pi_X(p))=0$. Hence, $\pi_X(p)=0$ which means that $p$ is compact. As $p$ is a projection, $p$ has finite-rank. As $\Psi$ is strongly continuous, we have that \[1_{\ell_2(X)}=p+\Psi(1_{\ell_2(Y)})=p+\SOTh\sum_{y\in Y}\Psi(e_{yy}).\] Therefore, Corollary~\ref{CorollaryMeasureURA} gives a partition $X=X'\sqcup X''$ and a map $f\colon X'\to Y$ so that 
$\inf_{x\in X'}\|\Psi(e_{f(x)f(x)})\delta_x\|>0$ and $\inf_{x\in X''}\|p\delta_x\|>0$. As $p$ has finite rank, $X'$ must be cofinite. By \cite[Lemma 6.3]{BragaFarahVignati2018}, replacing $X'$ by a smaller cofinite subset of $X$ if necessary, we can assume that $\inf_{x\in X'}\|\Phi(e_{xx})\delta_{f(x)}\|>0$.  The claim follows.
\end{proof} 

 By the previous claim, it follows immediately from \cite[Theorem 6.11]{BragaFarahVignati2018} that $X$ and $Y$ are coarsely equivalent.
 \end{proof}


\subsection{Ghosts}

Finally in this section, we step back from direct rigidity results, and prove Theorem~\ref{TGhostsPreserved} on ghost operators.

\begin{proof}[Proof of Theorem~\ref{TGhostsPreserved}]
Let $X$ and $Y$ be \ulf metric spaces and $\Phi\colon\cstu(X)\to \cstu(Y)$ be a $*$-isomorphism. We need to prove that $\Phi$ sends ghost operators to ghost operators. Proceeding as in the proof of Theorem~\ref{T1}, there are $\delta>0$ and a coarse equivalence $f\colon X\to Y$ such that $\|\Phi(e_{xx})\delta_{f(x)}\|>\delta$ for all $x\in X$. 

Suppose that $A\subseteq X$ is infinite. Since $X$ and $Y$ are \ulfx, the set $f[A]$ is also infinite. As $\|\Phi(\chi_A)\delta_{f(a)}\|>\delta$ for all $a\in A$, $\Phi(\chi_{A})$ cannot be a ghost. 

Now fix an arbitrary nonghost $a\in \cstu(X)$. Pick $\varepsilon>0$ and sequences of distinct elements $(x_n)_n$ and $(z_n)_n$ in $X$, such that $|\langle a\delta_{x_n},\delta_{z_n}\rangle |>\varepsilon$ for all $n\in\N$. Passing to subsequences if necessary and defining $A\coloneqq \{x_n\mid n\in\N\}$ and $B\coloneqq \{z_n\mid n\in\N\}$, we can assume that $\chi_Ba\chi_A-\SOTh\sum_ne_{z_nz_n}ae_{x_nx_n}$ is compact. Therefore, as $\chi_Ba\chi_A$ belongs to the ideal generated by $a$, so does $\SOTh\sum_ne_{z_nz_n}ae_{x_nx_n}$. As $|\langle a\delta_{x_n},\delta_{z_n}\rangle |>\varepsilon$ for all $n\in\N$, it follows that $b\coloneqq \SOTh\sum_ne_{z_nx_n}$ belongs to the ideal generated by $a$, and hence so does $\chi_A=b^*b$ (alternatively, \cite[Lemma 3.4]{ChenWang04} also implies that $\chi_A$ belongs to the ideal generated by $a$, even without going to subsequences). 
Since $\Phi$ is a $*$-isomorphism, $\Phi(\chi_A)$ belongs to the ideal generated by $\Phi(a)$. Since ghosts form an ideal, if $\Phi(a)$ is a ghost, then so is $\Phi(\chi_A)$. Hence, by the previous paragraph, $\Phi(a)$ is not a ghost. A symmetric argument gives that nonghost operators in $\cstu(Y)$ are mapped by $\Phi^{-1}$ to nonghost operators in $\cstu(X)$, and the conclusion follows.
\end{proof}

\section{Rigidity of stable Roe algebras and Morita equivalence}\label{SectionStable}

Given a \ulf metric space $X$ and an infinite-dimensional separable Hilbert space $H$, the \emph{stable Roe algebra of $X$} is defined by 
\[
\csts(X) \coloneqq \cstu(X)\otimes \cK(H),
\]
where the tensor product above is the minimal tensor product of \cstar-algebra theory. We can describe $\csts(X)$ more concretely as follows. For $x\in X$, let $v_x \colon H\to \ell_2(\{x\},H)\subseteq \ell_2(X,H)$ denote the canonical inclusion. For a bounded operator $a$ on $\ell_2(X,H)$ and $x,y\in X$, define the matrix entries $a_{xy} \coloneqq v_x^* a v_y\in \mathcal{B}(H)$, and define the propagation of $a$ to be 
\begin{equation*}
\propg(a) \coloneqq \sup\{d_X(x,y) \mid a_{xy}\neq 0\}\in [0,\infty]. 
\end{equation*}
 Given a finite-dimensional vector space $H'\subseteq H$ and $r>0$, let $\csts[X,r,H']$ denote the subspace of all operators $a=[a_{xy}]\in \cB(\ell_2(X,H))$ with propagation at most $r$ and such that each $a_{xy}$ is an operator in $\cB(H')$ (where $\cB(H')$ is identified with a \cstar-subalgebra of $\cB(H)$ in the canonical way). Then, under the canonical identification $\ell_2(X)\otimes H=\ell_2(X,H)$, the stable Roe algebra $\csts(X)$ is the norm-closure in $\cB(\ell_2(X,H))$ of the union of all such $\csts[X,r,H']$.
 
We will now show that stable uniform Roe algebras are also coarsely rigid (which in turn will give us Theorem~\ref{TMorita}).

\begin{theorem} \label{T1.Stable}
Suppose $X$ and $Y$ are \ulf metric spaces such that $ \csts(X)$ and $\csts(Y)$ are isomorphic. Then $X$ and $Y$ are coarsely equivalent. 
\end{theorem}
 
Before presenting the proof of Theorem~\ref{T1.Stable}, we need two lemmas.  The first is simple; we isolate it to keep the second lemma more transparent. 

\begin{lemma}\label{LemmaSpakulaWillettRiemann}
Let $X$ be a \ulf metric space, and let $(p_n)_{n\in \N}$ be a sequence of orthogonal projections in $\csts(X)$ such that 
 \[
 p_A \coloneqq \SOTh\sum_{n\in A}p_n
 \]
  is in $\csts(X)$ for all $A\subseteq \N$. Then, for all $\varepsilon>0$, there is a finite-rank projection $p\in \cB(H)$ such that $\| (1_{\ell_2(X,H)}-1_{\ell_2(X)}\otimes p)p_A\|\leq \varepsilon$ for all $A\subseteq \N$.
\end{lemma}

\begin{proof} 
As $p_\N$ is in $\mathrm{C}_s^*(X)$, there is a finite rank projection $p\in \cB(H)$ such that $\| (1_{\ell_2(X,H)}-1_{\ell_2(X)}\otimes p)p_\N\|\leq \varepsilon$.  For brevity, define $q\coloneqq (1_{\ell_2(X,H)}-1_{\ell_2(X)}\otimes p)$.  For any $A\subseteq \N$, $p_A\leq p_\N$; using this and the $\mathrm{C}^*$-equality
\[
\|qp_A\|=\|qp_Aq\|^{1/2}\leq \|qp_\N q\|^{1/2}=\|qp_\N\|\leq \varepsilon \eqno\qedhere
\]
	\end{proof}

Our next lemma is a stable version of Corollary~\ref{CorollaryMeasureURA}. Notice that the definition of $\varepsilon$-$r$-approximality (Definition~\ref{Def.e-m-approximated}) naturally extends to operators in $\cB(\ell_2(X,H))$. 

\begin{lemma}\label{LemmaMeasureURA.Stable}
Let $X$ be a \ulf metric space and $(p_n)_{n\in \N}$ be a sequence of orthogonal projections in $\csts(X)$ such that 
$p_A\coloneqq \SOTh\sum_{n\in A}p_n$ is in $\csts(X)$ for all $A\subseteq \N$. Then, for all unit vectors $\xi\in H$ with
\[
\inf_{x\in X}\|p_\N(\delta_x\otimes \xi)\|\geq  7/8,
\] 
we have that
\[
\inf_{x\in X}\sup_{n\in\N} \|p_n(\delta_x\otimes \xi)\|>0.
\]
\end{lemma}
 
\begin{proof}
Fix a unit vector $\xi\in H$ with 
\begin{equation}\label{over 7/8}
\inf_{x\in X}\|p_\N(\delta_x\otimes \xi)\|\geq 7/8.
\end{equation}
Lemma~\ref{LemmaUnifApprox} has a natural analog for stable Roe algebras (see \cite[Lemma 3.14]{BragaVignati20}), whence there is $r\geq 0$ such that each $p_B$ is $(1/8)$-$r$-approximable.\footnote{We will actually only need that $\|(\chi_{C}\otimes 1_H)p_B(\chi_{D}\otimes 1_H)\|\leq \varepsilon$ for all $C,D\subseteq X$ with $d(C,D)> r$ and all $A\subseteq \N$.}  Lemma \ref{LemmaSpakulaWillettRiemann} gives a finite rank projection $p\in \cB(H)$ such that  
\begin{equation}\label{p is big}
 \|(1_{\ell_2(X,H)}-1_{\ell_2(X)}\otimes p)p_B \|< 1/8 
\end{equation}
 for all $B\subseteq \N$. Define $N_r\coloneqq \sup_{x\in X}|B_r(x)|$ and $\delta\coloneqq (16N_r\rank(p))^{-1}$.  Fix some $x\in X$, and define $M:=\{n\in \N\mid \|p_n\delta_x\|<\delta\}$.  As $x$ is arbitrary and $\delta$ is fixed independently of $x$, to complete the proof it suffices to show that $M\neq \N$.  
 
Let $\pi:=\chi_{B_r(x)}\otimes p$, and consider the vector measure 
\[
\mu\colon\mathcal{P}(M)\to \ell_2(B_r(x))\otimes pH,\quad A\mapsto \pi p_A(\delta_x\otimes \xi).
\] 
As the $\mathbb{R}$-dimension of $\ell_2(B_r(x))\otimes pH$ is at most $2N_r \rank(p)$, Lemma \ref{lem:approx-infinite} and the choice of $M$ give us a subset $A$ of $M$ such that $\|\mu(A)-\frac{1}{2}\mu(M)\|\leq 2N_r\delta=1/8$.  In other words, we have that 
\[
\Big\|(1_{\ell_2(X)}\otimes p)(\chi_{B_r(x)}\otimes 1_H)\Big(p_A-\frac{1}{2}p_M\Big)(\delta_x\otimes \xi)\Big\|\leq \frac{1}{8}.
\]
As $p_A-\frac{1}{2}p_M=\frac{1}{2}p_A-\frac{1}{2}p_{M\setminus A}$ is $(1/8)$-$r$-approximable, this implies that 
\[
\Big\|(1_{\ell_2(X)}\otimes p)\Big(p_A-\frac{1}{2}p_M\Big)(\delta_x\otimes \xi)\Big\|\leq \frac{2}{8}.
\]
Hence line \eqref{p is big} applied with $B=A$ and $B=M$ implies that 
\[
\Big\|\Big(p_A-\frac{1}{2}p_M\Big)(\delta_x\otimes \xi)\Big\|<
\Big\|(1_{\ell_2(X)}\otimes p)\Big(p_A-\frac{1}{2}p_M\Big)(\delta_x\otimes \xi)\Big\|+\frac 3{16}\leq 
 \frac{7}{16}.
\]
As $p_A=p_Ap_M$, this and Lemma \ref{LemmaReferee} applied with $v=p_M(\delta_x\otimes \xi)$ imply that 
\begin{equation}\label{under 7/8}
\|p_M(\delta_x\otimes \xi)\|<7/8.
\end{equation}
On the other hand, from our initial assumption on $\xi$ from line \eqref{over 7/8} 
\[
\|(p_M+p_{\N\setminus M})(\delta_x\otimes \xi)\|\geq 7/8,
\]
so the above and line \eqref{under 7/8} imply that $p_{\N\setminus M}$ is non-zero.  Hence $\N\setminus M$ is non-empty, as required.
 \end{proof}
 
Before presenting the proof of Theorem~\ref{T1.Stable}, we isolate a result in the proof of \cite[Theorem 6.1]{SpakulaWillett2013AdvMath}, which is the analog of Proposition~\ref{PropOnceMapsFoundWeAreGood} in the setting of stable Roe algebras.

\begin{proposition}\label{PropOnceMapsFoundWeAreGood.Stable}
Let $X$ and $Y$ be \ulf metric spaces and $\Phi \colon \csts(X)\to\csts(Y)$ be a $*$-isomorphism. Suppose there are a finite-rank projection $p$ on $H$, a unit vector $\xi\in H$, and maps $f\colon X\to Y$ and $g\colon Y\to X$ such that 
\[
\inf_{x\in X}\|\Phi(\chi_{\{x\}}\otimes p_\xi)(\chi_{\{f(x)\}}\otimes p)\|>0
\] 
and 
\[
\inf_{y\in Y}\|\Phi^{-1}(\chi_{\{y\}}\otimes p_\xi)(\chi_{\{g(y)\}}\otimes p)\|>0,
\]
where $p_\xi$ is the projection on $H$ onto $\mathbb C\xi$. Then, $f$ is a coarse equivalence with coarse inverse $g$. \qed 
\end{proposition}
 
\begin{proof}[Proof of Theorem~\ref{T1.Stable}]
Let $\Phi \colon \csts(X) \to \csts(Y)$ be a $*$-isomorphism and let $\Psi=\Phi^{-1}$. We need to prove that $X$ and $Y$ are coarsely equivalent. Fix a unit vector $\xi\in H$ and let $p_{ \xi}$ be the projection of $H$ onto $\mathbb C\xi$. By Lemma~\ref{LemmaSpakulaWillettRiemann}, there is a finite-rank projection $p\in \cB(H)$ such that 
 \[
 \| (1_{\ell_2(Y,H)}-1_{\ell_2(Y)}\otimes p)\Phi(\chi_{\{x\}}\otimes p_\xi)\|<\frac{1}{8}
 \]
 for all $x\in X$. For each $y\in Y$, define $p_y\coloneqq \Psi(\chi_{\{y\}}\otimes p)$ and for each $A\subseteq Y$, define $p_A\coloneqq \SOTh\sum_{y\in A} p_y$. Then
\begin{align*}
\|(1_{\ell_2(X,H)}-p_Y)(\delta_x\otimes \xi)\|&=\|\chi_{\{x\}}\otimes p_\xi-\Psi(1_{\ell_2(Y)}\otimes p)(\chi_{\{x\}}\otimes p_\xi)\|\\
&=\|\Phi(\chi_{\{x\}}\otimes p_\xi)-(1_{\ell_2(Y)}\otimes p)\Phi(\chi_{\{x\}}\otimes p_\xi)\|\\
&< \frac{1}{8}
\end{align*}
for all $x\in X$. Hence Lemma~\ref{LemmaMeasureURA.Stable} gives $\delta>0$ and $f \colon X\to Y$ such that 
\[
\| \Phi( \chi_{\{ x \}} \otimes p_\xi ) ( \chi_{\{f(x)\}} \otimes p ) \| = \| \Psi( \chi_{\{f(x)\}} \otimes p ) ( \delta_x \otimes \xi ) \| > \delta,
\]
for all $y\in Y$.

By replacing $\delta$ by a smaller positive real and $p$ by a larger finite-rank projection if necessary, a symmetric argument gives $g\colon Y\to X$ such that 
\[
\| \Psi ( \chi_{\{y\}} \otimes p_\xi ) ( \chi_{\{g(y)\}} \otimes p ) \| = \| \Phi ( \chi_{\{g(y)\}} \otimes p ) ( \delta_y \otimes \xi ) \| > \delta
\]
for all $x\in X$. By Proposition~\ref{PropOnceMapsFoundWeAreGood.Stable}, $f$ is a coarse equivalence with coarse inverse $g$.
\end{proof}

\begin{proof}[Proof of Theorem~\ref{TMorita}]
The implication \eqref{Item1TMorita} $\Rightarrow$ \eqref{Item2TMorita} was established in \cite[Theorem 4]{BNW07}. Suppose now that $\cstu(X)$ and $\cstu(Y)$ are Morita equivalent. Then, the stable Roe algebras $\csts(X)$ and $\csts(Y)$ must be isomorphic as shown in \cite[Theorem 1.2]{BGR77}\footnote{See \cite[Chapter 7]{Lance95} for a shorter proof.}. It then follows from Theorem~\ref{T1.Stable} that $X$ and $Y$ must be coarsely equivalent.
\end{proof}

\begin{remark}\label{no roe rem}
The \emph{Roe algebra} $\cst(X)$ of a \ulf metric space $X$ is defined to be the norm closure of the set of finite propagation operators on $\ell_2(X,H)$ such that all matrix entries $a_{xy}$ are compact. Thus $\cst(X)$ is defined analogously to $\csts(X)$, but where the condition that all matrix entries be supported on the same finite-dimensional subspace of $H$ (up to an approximation) is dropped. It is known that if $X$ and $Y$ satisfy suitable geometric assumptions (see \cite[Theorem A]{LiSpakulaZhang2020} for the state of the art) then the following are equivalent:
\begin{enumerate}[(1)]
\item $X$ and $Y$ are coarsely equivalent;
\item $\cst(X)$ and $\cst(Y)$ are isomorphic.
\end{enumerate}
We do not know if this equivalence holds unconditionally: the techniques introduced in this paper seem to need some sort of `uniform local finite-dimensionality', which is not satisfied in the Roe algebra case.
\end{remark}

\section{Uniform Roe algebras of coarse spaces}\label{SectionCoarse}

Theorem~\ref{T1CoarseSpEmb} is established in this section. For that, we recall the basics of coarse spaces --- we refer the reader to the monograph \cite{Roe_book03} for a detailed treatment of coarse spaces. Given a set $X$ and a family $\cE$ of subsets of $X\times X$, $\cE$ is a \emph{coarse structure on $X$} if 
\begin{enumerate}[(1)]
\item $\Delta_X=\{(x,x)\mid x\in X\}\in\cE$,
\item $E\in \cE$ and $F\subseteq E$ implies $F\in \cE$,
\item $E,F\in \cE$ implies $E\cup F\in \cE$, 
\item $E\in \cE$ implies $E^{-1} \coloneqq \{(y,x)\mid (x,y)\in E\}\in \cE$, and 
\item $E,F\in \cE$ implies $E\circ F \coloneqq \{(x,y)\mid \exists z, (x,z)\in E\wedge (z,y)\in F\}\in \cE$.
\end{enumerate} 
The pair $(X,\cE)$ is then called a \emph{coarse space}, and the elements of $\cE$ are called \emph{controlled sets} (or \emph{entourages}).  A coarse space $(X,\cE)$ is \emph{uniformly locally finite} if for each $E\in \cE$ the cardinalities of the vertical and horizontal sections,
\[ 
E^x \coloneqq \{(x,y)\in E \mid y\in X\} \quad \text{ and } \quad E_y \coloneqq \{(x,y)\in E \mid x\in X\},
\]
are uniformly bounded. 

The motivating examples of coarse spaces are metric spaces. Indeed, if $(X,d)$ is a metric space, $X$ is endowed with the coarse structure 
\[
\cE_d\coloneqq \Bigg\{E\subseteq X\times X~\Big|~ \sup_{(x,y)\in E}d(x,y)<\infty\Bigg\}.
\]
A coarse space $(X,\cE)$ is called \emph{metrizable} if $\cE=\cE_d$ for some metric $d$ on $X$. It is well-known that $(X,\cE)$ is metrizable if and only if $\cE$ is countably generated and connected\footnote{In this context, \emph{countably generated} means that there is a countable collection $S$ of subsets of $X \times X$ such that $\cE$ is the intersection of all coarse structures containing $S$, and \emph{connected} means that $\{(x,y)\}\in \cE$ for all $x,y\in X$. The connectedness condition in the metric setting means that metrics are not allowed to take infinite values.} (see \cite[Theorem 2.55]{Roe_book03}). If $(X,\cE)$ and $(Y,\cF)$ are coarse spaces and $f\colon X\to Y$ is a map, $f$ is called \emph{coarse} if $(f\times f)[E]\in \cF$ for all $E\in \cE$, and $f$ is called \emph{expanding} if $(f\times f)^{-1}[F]\in \cE$ for all $F\in \cF$. A map $f\colon X\to Y$ which is both coarse and expanding is a \emph{coarse embedding}.  

The definition of uniform Roe algebras naturally extends to uniformly locally finite coarse spaces: we say then that $a=[a_{xy}]\in \cB(\ell_2(X))$ has \emph{controlled support} if $ \supp(a) \coloneqq \{(x,y)\mid a_{xy}\neq 0\}$ is in $ \cE$ and the \emph{uniform Roe algebra of $(X,\cE)$}, denoted by $\cstu(X,\cE)$, is the norm-closure of all operators on $\ell_2(X)$ with controlled support. For brevity, we often simply write $\cstu(X)$ for $\cstu(X,\cE)$.

The goal of this section is to prove Theorem~\ref{T1CoarseSpEmb}, which asserts that if $(X,\cE)$ and $(Y,\cF)$ are uniformly locally finite coarse spaces such that  $(X,\cE)$ is metrizable and $ \cstu(X)$ and $\cstu(Y)$ are $*$-isomorphic, then $Y$ is countable and $X$ coarsely embeds into $Y$.

The countability of $Y$ is straightforward.

\begin{lemma}
 Suppose that $X$ and $Y$ are coarse spaces. 
 
 If  $\cstu(X)$ and $\cstu(Y)$ are $*$-isomorphic, then $|X|=|Y|$. 
 
 If $\cstu(X)$ embeds into $\cstu(Y)$, then $|X|\leq |Y|$. 
\end{lemma}

\begin{proof} 
	The first part follows from the second, but it is also evident from the fact that an isomorphism between uniform Roe algebras is spatially implemented (\cite[Lemma 3.1]{SpakulaWillett2013AdvMath}, see Lemma~\ref{LemmaIsoImplementedUnitary}). 

If   $\cstu(X)$ embeds into $\cstu(Y)$  then $\ell_\infty(X)$ embeds into $\cB(\ell_2(Y))$. Since $\ell_\infty(X)$ has $|X|$ orthogonal nonzero projections, an orthonormal basis of $\ell_2(Y)$ has cardinality at least $|X|$, hence $|Y|\geq |X|$.  
 \end{proof}

On a related note,  it is possible that $\cstu(X)$ embeds into $\cstu(Y)$, where~$Y$ is metrizable and $X$ is a countable, nonmetrizable, coarse space; see  \cite[Proposition 6.5]{BragaFarahVignati2020AnnInstFour}.   


Before proving the nontrivial part of Theorem~\ref{T1CoarseSpEmb}, we need some preliminary results. The following notation will be used: given a $*$-isomorphism $\Phi\colon\cstu(X)\to \cstu(Y)$, $x\in X$, $y\in Y$, and $\eta>0$, define
\begin{itemize}
\item
$ X_{y,\eta} \coloneqq \{z\in X\mid \|\Phi^{-1}(e_{yy})\delta_z\|\geq \eta\}$, and 
 \item $Y_{x,\eta} \coloneqq \{z\in Y\mid \|\Phi(e_{xx})\delta_z\|\geq \eta\}$.
 \end{itemize}
The next lemma isolates a result in \cite{BragaFarahVignati2020AnnInstFour} which we will need later.
 
\begin{lemma}[{\cite[Lemma 4.7]{BragaFarahVignati2020AnnInstFour}}]
Let $(X,\mathcal{E})$ and $(Y,\mathcal{F})$ be \ulf coarse spaces, $\Phi\colon\cstu(X)\to \cstu(Y)$ be a $*$-isomorphism, and $f\colon X\to Y$ be such that $\inf_{x\in X}\|\Phi(e_{xx})\delta_{f(x)}\|>0$. The following hold:
\begin{enumerate}[(1)]
\item If for all $\eps>0$ there is $\eta>0$ such that 
\[
\|\Phi (e_{xx})(1_{\ell_2(Y)}- \chi_{Y_{x,\eta}})\|\leq \eps,
\] 
for all $x\in X$, then $f$ is expanding.
\item If for all $\eps>0$ there is $\eta>0$ such that 
\[
\|\Phi^{-1}(e_{f(x)f(x)})(1_{\ell_2(X)}- \chi_{X_{f(x),\eta}})\|\leq \eps,
\]
for all $x\in X$, then $f$ is coarse.\qed
\end{enumerate} \label{LemmaTheMapIsCoarse.GenCoarseSp}
\end{lemma} 

A simple application of Lemma~\ref{LemmaMeasureURA.Gen} gives:

\begin{corollary}\label{Cor.Expanding}
Let $X$ and $Y$ be \ulf coarse spaces and $\Phi\colon \cstu(X)\to \cstu(Y)$ be a $*$-isomorphism. If $X$ is metrizable, then for all $\varepsilon>0$ there is $\eta>0$ such that $\|\Phi (e_{xx})(1_{\ell_2(Y)}- \chi_{Y_{x,\eta}})\|\leq \eps$ for all $x\in X$.
\end{corollary} 
 
 \begin{proof}
Applying Lemma~\ref{LemmaMeasureURA.Gen} to the projections $(\Phi^{-1}(e_{yy}))_{y\in Y}$, we have that 
\[\lim_{\eta\to 0}\inf_{x\in X}\|\Phi(e_{xx})\chi_{Y_{x,\eta}}\|=\lim_{\eta\to 0}\inf_{x\in X}\|\Phi^{-1}(\chi_{Y_{x,\eta}})e_{xx}\|=1.\]
So, for all $\varepsilon>0$ there is $\eta>0$ such that for all $x\in X$ we have that $\|\Phi(e_{xx})\chi_{Y_{x,\eta}}\|>1-\varepsilon$. As each $\Phi(e_{xx})$ is a rank 1 projection (remember that $\Phi$ is rank-preserving), we have 
\[
1 = \| \Phi(e_{xx})1_{\ell_2(Y)} \|^2 = \|\Phi(e_{xx})(1_{\ell_2(Y)} - \chi_{Y_{x,\eta}}) \|^2 + \| \Phi(e_{xx})\chi_{Y_{x,\eta}}\|^2,
\]
and the result follows. 
 \end{proof}
 
We now present a technical lemma whose proof is inspired by \cite[Proposition 3.1]{Sako14} (cf. \cite[Proposition 2.4]{CTWY08}). In a sense, this lemma shows that a kind of operator norm localization property holds for arbitrary spaces (see \cite[Section 2]{CTWY08} for the definition of the operator norm localization property).
 
\begin{lemma}\label{LemmaSortOfONLForAnySpace}
Given $\varepsilon,\delta>0$, there is $\gamma>0$ such that for all $s,t>0$ there is $r>0$ for which the following holds: Let $X$ be a \ulf metric space, and let $p,q,a\in \cB(\ell_2(X))$, where $p$ is a projection and $q$ is a rank 1 projection. If $\propg(q)\leq t$, $\|p-a\|<\gamma $, $\propg(a)\leq s$, and $\|pq\|\geq \delta$, then there is $C\subseteq X$ with $\diam(C)\leq r$ such that $\|p\chi_C\|\geq 1-\varepsilon$.
\end{lemma}
 
\begin{proof}
Fix $\varepsilon,\delta>0$. Pick $k\in\N$ such that $(\delta/2)^{1/k}>1-\varepsilon$. Pick a positive $\gamma<(\delta/2)^{1/k}-1+\varepsilon$ small enough such that $\|p-a\|\leq \gamma$ implies $\|p-a^k\|\leq \delta/2$ for any projection $p$ and any operator $a$ in $ \cB(\ell_2(X))$. 

From now on, fix $s,t>0$, and $p,q,a\in \cB(\ell_2(X))$ as in the statement of the lemma. By our choice of $\gamma$, $\|p-a^k\|\leq \delta/2$. Hence, as $\|pq\|\geq \delta$, we have that $ \| a^kq\|\geq \delta/2$. Therefore, a telescoping argument implies that
\[ \prod_{i=0}^{k-1}\frac{\|a^{i+1}q\|}{\| a^iq\|}\geq \frac{\delta}{2}\]
(notice that $a^iq\neq 0$ for all $i$'s above). So, there is $j\in \{0,\ldots, k-1\}$ with 
\[
\|aa^{j}q\|\geq (\delta/2)^{\frac{1}{k}}\|a^{j}q\|.
\]
As $q$ is a rank 1 projection, we can pick a unit vector $\zeta\in \ell_2(X)$ such that $q=\langle\cdot,\zeta\rangle \zeta$. As $\propg(q)\leq t$, we have that $ \diam(\supp(\zeta))\leq t$. Define $\xi \coloneqq a ^{j} \zeta/\| a ^{j} \zeta\|$. As $\propg(a)\leq s$, it follows that \[\propg( a ^{j})\leq 2js +2s \leq 2ks.\]
Therefore, we must have that $\diam(\supp(\xi))\leq 4ks+t$.

At last, as $\|a\xi\|\geq (\delta/2)^{1/k}$, it follows that $\|p\xi\|\geq (\delta/2)^{1/k} -\gamma$. By our choice of $\gamma$, this shows that $ \|p\xi\|\geq 1-\eps$. The conclusion follows by letting $r=4ks+t$ and $C=\supp(\xi)$.
 \end{proof}

\begin{proof}[Proof of Theorem~\ref{T1CoarseSpEmb}]
Let $\Phi\colon\cstu(X)\to \cstu(Y)$ be a $*$-isomorphism and, for simplicity, let $\Psi=\Phi^{-1}$. We need to prove that $X$ coarsely embeds into $Y$. As $(X,d)$ is a metric space, Corollary~\ref{CorollaryMeasureURA} gives $\delta>0$ and $f\colon X\to Y$ such that $\|\Phi(e_{xx})\delta_{f(x)}\|>\delta$ for all $x\in X$.
 
By Lemma~\ref{LemmaTheMapIsCoarse.GenCoarseSp} and Corollary~\ref{Cor.Expanding}, $f$ is expanding. So, we are left to show that $f$ is coarse. 
 
Let $Z \coloneqq f(X)$ and pick $g \colon Z\to X$ such that $f(g(y))=y$ for all $y\in Z$. So, by our choice of $f$, it follows that 
\[
\|\Psi(e_{yy})\delta_{g(y)}\|=\|\Phi(e_{g(y)g(y)})e_{yy}\|=\|\Phi(e_{g(y)g(y)})\delta_{f(g(y))}\|>\delta,
\]
for all $y\in Z$. 

\begin{claim}\label{claim:aux}
For all $\varepsilon>0$ there is $r>0$ such that for all $y\in Z$, there is $C\subseteq X$ with $\diam(C)\leq r$, and such that $\|\Psi(e_{yy})\chi_{C}\|\geq 1-\varepsilon$.
\end{claim}

\begin{proof}
Fix $\varepsilon>0$ and let $\gamma>0$ be given by Lemma~\ref{LemmaSortOfONLForAnySpace} for $\varepsilon$ and $\delta$. As $X$ is metrizable, Lemma~\ref{LemmaUnifApprox} gives $s>0$ such that each $\Psi(e_{yy})$ is $\gamma$-$s$-approximable. Let $r>0$ be given by Lemma~\ref{LemmaSortOfONLForAnySpace} for $s$ and $t=0$. For each $y\in Z$, pick $a_z\in \cstu(X)$ with $\propg(a_z)\leq s$ such that $\|\Psi(e_{yy})-a_z\|\leq \gamma$. Since $\|\Psi(e_{yy})e_{g(y)g(y)}\|>\delta$ for all $y\in Z$, the result now follows from Lemma~\ref{LemmaSortOfONLForAnySpace}.
\end{proof}

\begin{claim}\label{claim:aux2}
For all $\varepsilon>0$, there is $\eta>0$ such that $\|\Psi(e_{yy})\chi_{X_{y,\eta}}\|\geq 1-\eps$, for all $y\in Z$.
\end{claim}

\begin{proof}
This follows from the proof of \cite[Lemma 6.7]{WhiteWillett20} or, equivalently, and with a more similar terminology, from \cite[Lemma 7.4]{BragaFarahVignati2018}. Indeed, in \cite[Lemma 7.4]{BragaFarahVignati2018} the metric spaces are assumed to have the operator norm localization property. However, an inspection of the proof reveals that the argument holds under the assumption that, for all $\varepsilon>0$, there is $r>0$ such that, for each $z\in Z$, there is $C\subseteq X$ with $\diam(C)\leq r$ satisfying $\|\Psi(e_{yy})\chi_C\|\geq 1-\varepsilon$. This statement is nothing else but Claim~\ref{claim:aux}. 
\end{proof}

As each $\Psi(e_{yy})$ has rank 1, Claim~\ref{claim:aux2} implies that for all $\varepsilon>0$, there is $\eta>0 $ such that $\|\Psi(e_{yy})(1_{\ell_2(X)}- \chi_{X_{y,\eta}})\|\leq \eps$ for all $y\in Z$ (cf. the proof of Corollary~\ref{Cor.Expanding}). By Lemma~\ref{LemmaTheMapIsCoarse.GenCoarseSp}, we conclude that $f$ is coarse.
\end{proof} 
\begin{acknowledgments}
This paper was written under the auspices of the American Institute of Mathematics (AIM) SQuaREs program and as part of the `Expanders, ghosts, and Roe algebras' SQuaRE project. F.\ B.,\ B.\ M.\ B. \ and R.\ W.\ were partially supported by the US National Science Foundation under the grants DMS-1800322 and DMS-2055604, DMS-2054860, and DMS-1901522, respectively. I.\ F.\ was partially supported by NSERC. A.\ V.\ is supported by an `Emergence en Recherche' IdeX grant from the Universit\'e de Paris and an ANR grant (ANR-17-CE40-0026). Last, but not least, we are indebted to the anonymous referee for a thoughtful  report: in particular, the referee suggested that Lemma~\ref{lem:approx-infinite} can be proved using \cite{Lindenstrauss1971}, and provided a better insight into the proof of our Lemma~\ref{LemmaMeasureURA.Gen} that resulted in substantial simplification in this, and also of Lemma~\ref{LemmaMeasureURA.Stable}.

\end{acknowledgments}

\bibliographystyle{plain}
\bibliography{bibliography}
\end{document}